\label{key}
\pdfoutput=1
\documentclass[preprint, noinfoline,addressatend,imslayout]{imsart}

\usepackage[top=20mm, bottom=20mm, left=15mm, right=15mm]{geometry}

\usepackage{enumerate}
\usepackage{url}
\usepackage{marg}

\RequirePackage[colorlinks,citecolor=blue,urlcolor=blue]{hyperref}

\usepackage[authoryear]{natbib}
\usepackage{multirow}

\usepackage{amsthm,amsmath,amsfonts,amssymb}
\usepackage{enumerate}

\usepackage{bm,accents}
\usepackage{graphicx}
\usepackage{nicefrac}

\usepackage{color}
\startlocaldefs

\newcommand{\Fn}{\widehat{F}}

\newcommand{\rn}{\widehat{r}}
\newcommand{\dldp}{\frac{\partial{l(D;\theta,\gamma)}}{\partial{\gamma}}}

\newcommand{\dldt}{\frac{\partial{l_{partial}(D;\theta,\gamma)}}{\partial{\gamma}}}

\newtheorem{thm}{Theorem}

\newtheorem{lem}{Lemma}

\startlocaldefs

\setattribute{tablecaption}{shape}{}

\setattribute{tablename} {skip}{.~}

\bibpunct{(}{)}{;}{a}{,}{,}
\def\spacingset#1{\renewcommand{\baselinestretch}%
	{#1}\small\normalsize} \spacingset{1}

\begin{document}

	\begin{frontmatter}
		

		\title{Self-reporting and screening: Data with  current-status and censored observations}
		\maketitle
		\runtitle{Self-reporting and screening}
		
		\begin{aug}
			\author[1]{\fnms{Jonathan} \snm{Yefenof}\ead[label=e1]{jonathan.yefenof@mail.huji.ac.il}}
			 \author[2]{\fnms{Yair}
				\snm{Goldberg}\ead[label=e2]{ygoldberg@stat.haifa.ac.il}}
			\author[3]{\fnms{Jennifer}
				\snm{Wiler}\ead[label=e3]{ygoldberg@stat.haifa.ac.il}}
			 \author[4]{\fnms{Avishai}
			 	\snm{Mandelbaum}\ead[label=e3]{ygoldberg@stat.haifa.ac.il}}
			 and \author[1,5]{\fnms{Ya'acov}
			 	\snm{Ritov}\ead[label=e1]{ygoldberg@stat.haifa.ac.il}}
			
			\affiliation[1]{The Hebrew University of Jerusalem}
				\affiliation[2]{University of Haifa}
				\affiliation[3]{University of Colorado}
			\affiliation[4]{Technion - Israel Institute of Technology}
		\affiliation[5]{University of Michigan}
		
			

			\runauthor{Yefenof et al.}
		\end{aug}
		\begin{abstract}
			
		We consider survival data that combine three types of  observations: uncensored, right-censored, and left-censored. Such data arises from screening a medical condition, in situations where self-detection arises naturally. Our goal is to estimate the failure-time distribution, based on these three observation types. We propose a novel methodology for distribution estimation using both parametric and nonparametric techniques.  We then evaluate the performance of these estimators via simulated data. Finally, as a case study, we estimate the patience of patients who arrive at an emergency department and wait for treatment. Three categories of patients are observed: those who leave the system and announce it, and thus their patience time is observed; those who get service and thus their patience time is right-censored by the waiting time;  and those who leave the system without announcing it. For this third category, the patients' absence is revealed only when they are called to service, which is after they have already left; formally, their patience time is left-censored. Other applications of our proposed methodology are discussed.

		\end{abstract}

		
	\end{frontmatter}

\section{\label{sec:Introduction}Introduction}
We study the estimation of failure time distribution where the failure times can be either observed directly, or be right-censored or left-censored. This type of survival data arises, for example, in estimation of time to the appearance of a medical condition where characteristic symptoms may or may not appear when the condition exists. Specific medical settings include relapse in childhood brain tumors, which may be observed due to clinical symptoms, or right-censored due to periodic screening with negative result (no tumor), or left-censored due to periodic screening with a positive result  \citep{minn2001surveillance}. Another medical setting is melanoma cancer, which is observed if self-detected, or is right censored due to a negative screening (no melanoma), or  left-censored if it goes undetected until screening. Additional examples  can be found in \citet{whitehead1989analysis}.

The motivating example for this work comes from estimating customer patience in service system which, as discussed by \citet{mandelbaum2007service}, is a challenging problem. In our study, we focus on patients who wait for treatment in an emergency department (ED). Three categories of patients are observed. The first category consists of patients who get service and thus their patience time is right-censored by the waiting time. The second category comprises those who leave the system and announce it, and thus their patience time is observed while the waiting time is right-censored. The third category consists of patients who leave the system without announcing it; their absence is hence revealed only when they are called to service, which is after they have already left; formally, their patience time is left-censored.

Estimating the patience time is of importance as the  decision of patients to leave the system before getting served might have a strong effect on their physical well-being. There has been considerable research on the reasons why patients leave an ED before being served; see \citet{baker_patients_1991}, \citet{hunt_characteristics_2006}, \citet{bolandifar2014modeling}, and \citet{batt2015waiting}. However, these and other authors have not proposed a model by which ED patience time - namely the duration that a potential patient is willing to wait for ED service - can be estimated, and this is our goal here.

We propose novel parametric and nonparametric estimators of the unknown survival function for this 3-type survival data. We then study their rates of convergence.
The parametric estimator is based on both full and partial likelihoods. We provide condition under which the parametric estimator is a linear asymptotic normal (LAN) estimator and converges to a normal distribution in a root-$n$ rate.	The nonparametric estimator is based on nonparametric kernel estimators for density functions and on a novel estimator of the cumulative probability function that has some similarities to the Nelson--Aalen estimator \citep[e.g.,][Chapter 4]{klein_survival_2013}.
We show that, under some regularity conditions, the nonparametric estimator point-wise converges to the normal distribution.

We perform a simulation study and compare the proposed parametric and nonparametric estimators. For the parametric model, we study both correct and misspecified models and show the different corresponding results. We show how the accuracy changes with  sample size. We then carry out a case study that is based on data of patients waiting for treatment in an ED, in the U.S.\ in 2008.
We analyzed separately different severity levels (15106  observations in the emergency group, 43600 in the urgent group, and 26541 in the semi-urgent group). We conclude with a comparison of the parametric and nonparametric estimators for the three different severity levels of this dataset. 

\section{Brief Literature Review}\label{sec:Literature}
Developing screening methods for medical conditions, such as breast and melanoma cancers, has a long history \citep{wilson1968principles,zelen1969theory}. In the classical setting, the medical condition either already exists at the time of screening and is thus left-censored, or does not exist, and is thus right-censored. The setting in which  self-detection is possible, and thus the condition time is observed, has been surprisingly  mostly ignored in the literature. For example, \citet{minn2001surveillance} treat both self-detection times and screening times as event times, ignoring the censoring. The closest model to the one that we  present here appears in \citet{whitehead1989analysis}. It is assumed there that the condition can be detected at screening or before screening due to symptoms. In both cases, the condition already exists at the time of detection. It is also assumed that screenings take place at a sequence of fixed time points. \citet{whitehead1989analysis} recommends to ignore the extra knowledge gained due to self-reporting and to replace these times with the time of the next screening. The survival function is then estimated only at the discrete fixed screening times using standard techniques \citep{prentice1978regression}.

There has been considerable research effort, dedicated to modeling and analysis of customer (im)patience while waiting for service. Here we describe several papers that, together with references therein, provide what is required for a historical background and state-of-art perspective. First, we recommend the literature review (Section 3) in the recent \cite{batt2015waiting}, accompanied by \citet{gans2003telephone}: these survey patience-research from an operational/queueing view point (mainly Section 6.3.3 in the latter), while connecting it to the medical literature on LWBS (mainly Section 3 in the former); see also \citet{aksin2007modern} who, relative to \citet{gans2003telephone}, expand on managerial challenges.  Next we mention \citet{mandelbaum2013data}, which is an Explanatory Data Analysis of (im)patience in telephone call centers (that appears in a special issue that is devoted to models of queues abandonment). Finally, and the most related to the present study, are the following two studies. \citet{brown2005statistical} applies, in Section 5, the Kaplan--Meier estimator \citep{kaplan1958nonparametric} to estimate the survival functions and consequently hazard rates, of both virtual waiting time and impatience; the data is that of a call center, in which times of abandonment are all recorded hence the data is right-censored. Then \citet{wiler2013emergency}, which is also the source of our present ED data case study, estimate LWBS rates as a function of ED patient arrival rates, treatment times, and ED boarding times. There was no attempt in that work to estimate the patience-time distribution.

We conclude this brief survey with the observation that the estimation of customer (im)patience is relevant beyond screening, call centers, and EDs. For example, \citet{nah2004study} studies tolerance of Web users (during information retrieval). \citet{YomTov2017} analyzes chat services, in which customers abandon at any phase during chat-exchanges with a service center: one expects that such services give rise to the same options as in EDs: some customers receive service, others abandon without letting anyone know, and the rest announce their abandonment time.   

	\section{The Model}\label{sec:Model}

In the standard setting of right-censored data one observes, for each patient, either the failure time or the censoring time. In terms of our motivating example, failure time is patience time while censoring time is the waiting time. Patience time is observed when patients leave the ED while informing the system of their departure; waiting time is observed when a patient is called for service. However, unlike in standard right-censored data and like in current status data, there are also patients who leave without informing; in this case their absence is observed only when they are called for service, and this latter time provides an upper bound for their patience time.  In other words, the (virtual) waiting time is observed, and the only information on patience time is that it is less than this observed waiting time.  Hence, in this case, the patience time is left-censored.

More formally, let $T$ be the patient's failure time, i.e., the time until the patient loses patience. Let $W$ be the censoring time, i.e., the waiting time until the patient gets (or could have gotten) service. We assume that $T$ has a cumulative distribution function (cdf) $F$ and a probability density function (pdf) $f$, and that $W$ has cdf $G$ and pdf  $g$.  Let $\Delta$ be the indicator  $\Delta\equiv1{\{T< W\}}$; i.e., $\Delta=1$ if the patient loses patience before being called to service, and $\Delta=0$ otherwise.

Let $Y$ be the indicator that is $1$ for a patient who leaves and informs when leaving, and $0$ otherwise. Denote by $q(t)$ the conditional probability that a patient reports leaving given that the waiting time equals to $t$. In other words, $q(t)=pr(Y=1\mid{T=t})$.
We assume that the waiting time $W$ and the patience time $T$ are independent. This assumption, which is common in the right-censored data literature \citep[see,][Chapter 3, pages 65-66]{klein_survival_2013}, seems appropriate in our case study, as we stratify by acuity levels.
We also assume that announcement indicator $Y$ is independent of the waiting time $W$, as it seems reasonable that the decision of a patient to report when leaving does not depend on the waiting time. 
Summarizing, we assume that the pair $(Y,T)$ is independent of the waiting time $W$. When this assumption does not hold, different theoretical tools are needed for a valid estimation. 

Let $U$ be the recorded time: $U\equiv{Y}{T}+(1-Y){W}$. The observed data consist of the triplets $(U_i,Y_i,\Delta_i)$, $i=1,\ldots,n$, and there are three categories of patients:

\begin{description}
	\item[$\mathcal{C}=1$:] The patient gets service, hence the waiting time is observed, which serves as a lower bound on the patience time; thus the patience time is right censored. Formally,  $\Delta=0$, $Y=0$, and $U=W$.
	\item[$\mathcal{C}=2$:] The patient leaves without being treated and reports departure. The patience time is thus revealed:  $Y=1$, $\Delta=1$, and $U=T$.
	\item[$\mathcal{C}=3$:] The patient leaves without reporting, hence virtual waiting time (the time that the patient would have waited had he stayed in the ED) is observed, which provides an upper bound for the patience time, thus the patience time is left-censored. Formally, $Y=0$, $\Delta=1$, and  $U=W$.

\end{description}

\begin{lem}\label{lem:probablities_connections}
	The following equalities hold:
	\begin{enumerate}[i)]
		\item\label{a}
		$pr(U\leq{t},\mathcal{C}=1)=\int_{0}^{t}g(w)\overline{F}(w)dw$.
		
		\item \label{b}
		$pr(U\leq{t},\mathcal{C}=2) =\int_{0}^{t}q(w)f(w) \overline{G}(w)dw$.
		
		\item\label{c}
		$pr(U\leq{t},\mathcal{C}=3) = \int_{0}^{t} g(w)\int_{0}^{w} \left\lbrace1-q(x)\right\rbrace f(x)dxdw$.

	\end{enumerate}
\end{lem}
See the proof in \ref{l:2}.

For  $i=1,2,3$, we introduce the following sub-stochastic distribution functions
\begin{align}\label{key}
h_{i}(t):=\frac{d}{dt}pr(U\leq{t},\mathcal{C}=i).
\end{align}
From Lemma \ref{lem:probablities_connections} above, we deduce that
\begin{align*}
h_1(t)= g(t)\overline{F}(t),\quad\ h_{2}(t)=q(t)f(t)\overline{G}(t),\quad\ h_{3}(t)=g(t)\int_{0}^{t}\left\lbrace{}1-q(x)\right\rbrace{}f(x)dx.
\end{align*}
Here, $\overline{F}(t)=1-F(t)$ and $\overline{G}(t)=1-G(t)$ are the survival functions of the patience time and the waiting time, respectively.

Define
\begin{align}\label{def}
r_{1}(t)\equiv \frac{h_{1}(t)}{pr(W\leq{T})},\quad r_{2}(t)\equiv\frac{h_{2}(t)}{pr(Y=1,W>T)},\quad
r_{3}(t)\equiv \frac{h_{3}(t)}{pr(Y=0,W>T)}.
\end{align}
Then  $r_{i}$ is the density function of the observed time $U$ given $\mathcal{C}=i$. Our model assumes that all denominators are positive. 

To summarise what is known and what is to be estimated,  there are two unknown distributions in our setting, $G$ and $F$, and we aim to estimate them using both parametric and nonparametric techniques. For each patient, the waiting time is either observed or right censored. If the patient reports and then leaves, the waiting time is longer than the observed patience time. Hence, the waiting time is right-censored. Therefore, parametric and nonparametric estimation for the distribution of waiting time $W$ can be done by standard techniques for right-censored data. However, estimation of the distribution of patience time $T$, is more complicated and is discussed in Sections \ref{sec:par} and \ref{sec:nonpar}.

	\section{Parametric estimation}\label{sec:par}
Assume now that the distributions of both the patience time and the waiting time belong to some parametric families. More formally, let $\mathcal{F}=\{f(\cdot;\theta), \theta\in\Theta\}$ where $\Theta\subseteq\mathbb{R}^d$, $\mathcal{G}=\{g(\cdot;\gamma), \gamma\in{\Gamma}\}$ where $\Gamma\subseteq \mathbb{R}^p$. We assume that the density of the patience time can be written as $f(t;\theta_{0})\in\mathcal{F}$. We also assume that the density of the waiting time can be written as $g(t;\gamma_{0})\in\mathcal{G}$.
Write
$h_1(t;\theta,\gamma)\equiv{g}(t;\gamma)\overline{F}(t;\theta)$,
and similarly $h_2(t;\theta,\gamma)\equiv{q(t)}{f}(t;\theta)\overline{G}(t;\gamma)$ and $h_3(t;\theta,\gamma)\equiv{g}(t;\gamma)\int_{0}^{t}\left\lbrace{}1-q(x)\right\rbrace{}f(x;\theta)dx$.

The likelihood of the observed data $D=\{(U_i,Y_i,\Delta_i), i=1,\ldots,n\}$ can be written in terms of the functions $h_1$, $h_2$, and $h_3$, as follows:
\begin{align*}
L(D;\theta,\gamma)= \prod_{i=1}^{n}\left\lbrace{h}_{1}(U_{i};\theta,\gamma)\right\rbrace{ }^{1-\Delta_{i}}\left\lbrace{}h_{2}(U_{i};\theta,\gamma)\right\rbrace {}^{\Delta_{i}Y_{i}}\left\lbrace{}h_{3}(U_{i};\theta,\gamma)\right\rbrace {}^{\Delta_{i}(1-Y_{i})}.\end{align*}
Using the explicit representations of $h_1$, $h_2$, $h_3$, we obtain that $L(D;\theta,\gamma)$ is given by
\begin{align*}
\prod_{i=1}^{n}&\Big(\left\lbrace{}g(U_{i};\gamma)\overline{F}(U_{i};\theta)\right\rbrace{}^{1-\Delta_{i}}\left\lbrace{}q(U_{i})f(U_{i};\theta)\overline{G}(U_{i};\gamma)\right\rbrace{}^{\Delta_{i}Y_{i}}
\\
&\times \left[ g(U_{i};\gamma)\int_{0}^{U_{i}}{\left\lbrace{} 1-q(s)\right\rbrace f(s;\theta)ds}\right] ^{\Delta_{i}(1-Y_{i})}\Big).
\end{align*}
The value of $\gamma$ that maximizes this likelihood is independent of $\theta$. Therefore, a maximum likelihood estimator (MLE) $\widehat{\gamma}_{n}$ to $\gamma_{0}$ can be constructed from this likelihood.
However, maximizing the likelihood with respect to $\theta$ is difficult. Even if $\gamma$ is given or estimated, the maximizer of $\theta$ depends on the unknown function $q(t)$.  Therefore, we consider the partial likelihood  $L_{partial}(D;\theta;\gamma)$ of category $\mathcal{C}=1$,
\begin{align*}\prod_{i=1}^{n}\left\lbrace {}\frac{g(U_{i};\gamma)\overline{F}(U_{i};\theta)}{\int_{0}^{\infty}g(s;\gamma)\overline{F}(s;\theta)ds}\right\rbrace{^{1-\Delta_{i}}}.
\end{align*}
The value of $\theta$ that maximizes this partial likelihood depends on $\gamma$. We plug the MLE $\widehat{\gamma}_{n}$ into this partial likelihood. In Theorem~\ref{Thm1} below we show that, under standard regularity conditions, the maximizer $\widehat\theta_{n}$ of  $L_{partial}(D;\theta;\widehat\gamma_{n})$ is a consistent and asymptotically normal estimator for $\theta_{0}$.

We need the following assumptions:
\begin{enumerate}[({A}1)]
	\item \label{A1} The derivative $\frac{\partial}{\partial{\theta}}{f(t;\theta)}$ is continuous in $t$
	for each $\theta\in\Theta$, $\frac{\partial}{\partial{\gamma}}{g(t;\gamma)}$ is continuous in $t$ for each $\gamma\in\Gamma$.
	\item \label{A2} For all $\theta\in\Theta$,
	$\arg\max_{\gamma\in\Gamma}{L(D;\theta,\gamma)}$ is unique, hence denote \\$\widehat{\gamma}(\theta)\equiv\arg\max_{\gamma\in\Gamma}{L(D;\theta,\gamma)}$. It is assumed as well that for each $\theta\in\Theta$,  $\frac{\partial}{\partial\gamma}L\left\lbrace D;\theta,\widehat{\gamma}(\theta)\right\rbrace =0$.
	\item \label{A3}For all $\gamma\in\Gamma$, $\arg\max_{\theta\in\Theta}{L_{partial}(D;\theta,\gamma)}$ is unique, hence denote \\$\widehat{\theta}(\gamma)\equiv\arg\max_{\theta\in\Theta}{L_{partial}(D;\theta,\gamma)}$.
	It is assumed as well that for each $\gamma\in\Gamma$,  $\frac{\partial}{\partial\theta}L_{partial}\left\lbrace D;\widehat{\theta}(\gamma),\gamma\right\rbrace =0$.
	
\end{enumerate}
\begin{thm}\label{Thm1}
	Let $\widehat{\gamma}_{n}$ be the maximizer of $L(D;\theta;\gamma)$ and let $\widehat{\theta}_{n}$ be the maximizer of $L_{partial}(D;\theta;\widehat\gamma_{n})$. Then, as $n\rightarrow{\infty}$,
	\begin{enumerate}[i)]
		\item \label{enum:consistent_gamma}$ \widehat{\gamma}_{n}\rightarrow{\gamma_{0}}$ in probability.
		\item \label{enum:normal_gamma}$ \sqrt{n}\left(\widehat{\gamma}_{n}-\gamma_{0}\right)\rightarrow{N\left(0,V_{\gamma_{0}}\right)}$ in distribution.
		
		\item \label{enum:consistent_theta}$\widehat{\theta}_{n}\rightarrow{\theta_{0}}$ in probability.
		\item \label{enum:normal_theta} $\sqrt{n}\left(\widehat{\theta}_{n}-\theta_{0}\right)\rightarrow{N\left(0,S_{\theta_{0},\gamma_{0}}\right)}$ in distribution.
	\end{enumerate}
	Here $V_{\gamma_{0}}$, $S_{\theta_{0}},\gamma_{0}$ are covariance matrices as defined in Appendix~\ref{l:2}. 	
\end{thm}

The proof appears in Appendix~\ref{l:2}.

	\section{Nonparametric estimation}\label{sec:nonpar}

In this section we propose nonparametric estimators for the survival function of the patience time $\overline{F}$ and study its theoretical properties.
For simplicity, we restrict the estimation to a segment $[0,\tau]$ for some $\tau>0$, such that the probability of $W$ and $T$ being larger than $\tau$ is positive. This is a standard condition in survival estimation  \citep[see][Chapter 4.2]{Kosorok08}.
Note that for observations of Categories 1 and 3, the waiting-time is observed. For Category 2, only a lower bound of the waiting time is observed. Hence, the waiting time is either observed or right-censored. Therefore, estimating the waiting time distribution can be done by using standard survival analysis estimators such as the Kaplan--Meyer estimator \citep[see][Chapter 4]{klein_survival_2013}.  On the other hand, estimating the distribution of the patience time is more challenging since we cannot distinguish between the density function $f$ and the unknown function $q$.
Our goal is thus to estimate the distribution of the patience time F.

Assume that over all positive numbers, the waiting time density function $g$ is strictly positive. Recall that $h_1(t)=g(t)\overline{F}(t)$, $h_{3}(t)=g(t)\int_{0}^{t}{\left\lbrace{}1-q(s)\right\rbrace f(s)ds}$,
where the functions $h_{1}, h_{3}$ are defined as in (\ref{key}). Therefore, \begin{align}\label{eqq}\frac{h_{1}(t)}{h_{3}(t)}=\frac{\overline{F}(t)}{F(t)-\int_{0}^{t}{q(s)f(s)ds}}.
\end{align}
which is well defined as $g(t)>0$.
Reordering the terms in (\ref{eqq}), we get that
\begin{align*}
\left\lbrace{}{F(t)-\int_{0}^{t}{q(s)f(s)ds}}\right\rbrace\frac{h_{1}(t)}{h_{3}(t)}={1-{F}(t)}.
\end{align*}
Hence,
\begin{align*}
F(t)=\frac{h_{3}(t)+h_{1}(t)\int_{0}^{t}{q(s)f(s)ds}}{h_{3}(t)+h_{1}(t)}.
\end{align*}
From the definitions in (\ref{def}), it follows that
\begin{align}\label{eq:estimator}
F(t)=\frac{pr(Y=0,T<W)r_{3}(t)+pr(W\leq{T})r_{1}(t)\int_{0}^{t}{q(s)f(s)ds}}{pr(Y=0,T<W)r_{3}(t)+pr(W\leq{T})r_{1}(t)}.
\end{align}
Therefore, we propose to estimate {F(t)} by estimating the following terms:

(i) $pr(W\leq{T})$ and $pr(Y=0,T<W)$,

(ii) $r_{1}(t)$ and $r_{3}(t)$,

(iii) $A(t)\equiv\int_{0}^{t}{q(s)f(s)ds}$.

Estimating the expression in (i) can be done by the empirical estimators:
$\widehat{pr}(T\leq{W})=n^{-1}\Sigma_{i=1}^{n}(1-\Delta_{i})$, $\widehat{pr}(Y=0,W<T)=n^{-1}\Sigma_{i=1}^{n}\Delta_{i}(1-Y_{i})$.
These estimators converge, by the central limit theorem (CLT), to $pr(W\leq{T})$ and $pr(Y=0,T<W)$, respectively, at the rate of $n^{1/2}$.

Since $r_{1}$ and $r_{3}$ are density functions, they can be estimated using a kernel estimator \citep[Chapter 1.2]{tsybakov_introduction_2008}. Let $\widehat{r}_{1}$ and $\widehat{r}_{3}$ be kernel estimators of $r_{1}$ and $r_{3}$, respectively. Assume that both $r_{1}$ and $r_{3}$ belong to a Sobolev function class of order $\beta$. Then for each $t>0$, both $\widehat{r}_{1}(t)$ and  $\widehat{r}_{3}(t)$ converge at a rate of $n^{\beta/(2\beta+1)}$  \citep[see][Chapter 1.7, for both the definition of a Sobolev class and the proof]{tsybakov_introduction_2008}.

We now turn to estimate the term $A(t)={\int_{0}^{t}{q(s)f(s)ds}}$. A nonparametric estimator that we created for this term is defined and proven to be consistent in the following lemma.
\begin{lem}\label{lem:estimating}
	Let \begin{align*}   \widehat{N}_{n}(t)\equiv\frac{1}{n}\sum_{i=1}^{n}Y_{i}\Delta_{i}1\{U_{i}\leq{t}\},\quad \widehat{Y}_{n}(t)\equiv\frac{1}{n}\sum_{i=1}^{n}1\{U_{i}\geq{t}\}\,.\end{align*}
	Define $\widehat{D}_{n}(t)\equiv\int_{0}^{t}\frac{d\widehat{N}_{n}(s)}{\widehat{Y}_{n}(s)}$.
	Then $\widehat{A}(t)\equiv1-\exp\left\lbrace {-\widehat{D}_{n}(t)}\right\rbrace $ converges pointwise to $A(t)$, at a rate of $n^{1/2}$, for every $t\in{[0,\tau]}$.\end{lem} The proof is given in Appendix~\ref{p:3}.

By plugging in the estimators 
\begin{align*}\widehat{pr}(Y=0,W<T),\, \widehat{pr}(T\leq{W}),\, {\rn}_{3}(t),\, {\rn}_{1}(t),\, \widehat{A}(t),\end{align*}
to the equation in (\ref{eq:estimator}), we get that \begin{align}\label{eq:nonparametric_estimator}
\Fn_{n}(t)=\frac{\widehat{pr}(Y=0,W<T){\rn}_{3}(t)+\widehat{pr}(T\leq{W}){\rn}_{1}(t)\widehat{A}(t)}{\widehat{pr}(Y=0,W<T){\rn_{3}}(t)+\widehat{pr}(T\leq{W})\rn_{1}(t)},\end{align} is an estimator of $F(t)$. 

\begin{thm}\label{thm2}
	The estimator $\hat{F}_{n}(t)$ converges pointwise to $F(t)$ at a rate of $n^{{\beta}/{(2\beta+1)}}$, \\for every $t\in{[0,\tau]}$.
\end{thm}
The proof appears in Appendix~\ref{p:2}.

\begin{figure}[t!]
	\centering
	\includegraphics[width=\textwidth]{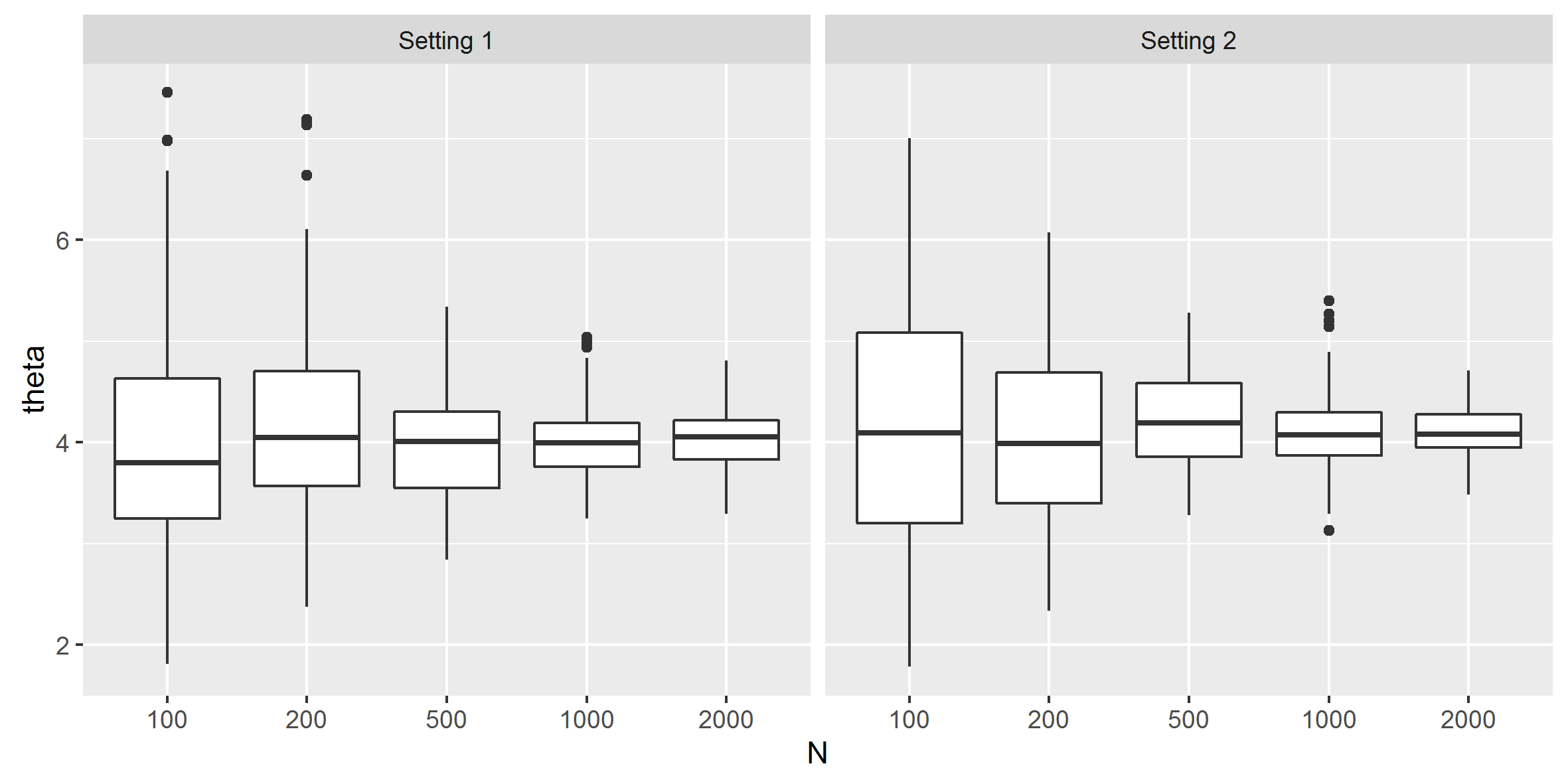}\protect\caption{Parametric estimation. Setting 1: $T\sim{Exp(4)}$ and $W\sim{Exp(10)}$, the parametric estimator converges to the true rate $4$. Setting 2: $T\sim{Weibull(4,2)}$ and $W\sim{Exp(10)}$, the parametric estimator does not converge to the true rate $4$.\label{figure:sim-par}}
\end{figure}

\section{Simulations}\label{sec:sim}
\begin{table}[!b]
	
	\centering
	\begin{small}
		\begin{tabular}{lccccccccccccc}
			&\multicolumn{6}{c} {Setting 1}&& \multicolumn{6}{c}{Setting 2}\\					
			&  \multicolumn{3}{c} {Parametric}& \multicolumn{3}{c}{Nonparametric}&&  \multicolumn{3}{c} {Parametric}& \multicolumn{3}{c}{Nonparametric}\\
			N & mean & median & sd & mean & median & sd && mean & median & sd & mean & median & sd \\[5pt]
			100 & 6.23	& 2.79	& 8.2		& 14.41	& 12.64	& 8.28 &&23.24	& 18.44	& 11.61		& 16.47	& 13.18	& 12.16 \\
			200 & 3.29	& 1.32	& 4.94   	& 9.77	& 7.48	& 6.51 & & 19.07	& 15.67	& 7.87   	& 10.43	& 8.47	& 7.33 \\
			500	& 1.31	& 0.77	& 1.57   	& 5.32	& 4.74	& 2.49 &&16.53	& 15.01	& 4.09   	& 4.33	& 3.95	& 2.18 \\
			1000	& 0.72	& 0.25	& 0.99   	& 3.55	& 3.36	& 1.32&&15.38	& 14.16	& 3.68   	& 2.49	& 2.2	& 1.18\\ 
			2000	& 0.36	& 0.16	& 0.53   	& 2.35	& 2.21	& 0.87&	& 14.8	& 14.18	& 1.98   		& 1.78	& 1.65	& 0.9
		\end{tabular}
	\end{small}

	\caption{MSE for Settings 1 and 2. The table summarizes the MSE that was calculated (100 times) for each of the sample sizes.
		For Setting~1, the patience time  and waiting time are distributed $T\sim{Exp(4)}$ and the  $W\sim{Exp(10)}$, respectively. In Setting~2 the patience time  and waiting time are distributed  $T\sim{Weibull(4,2)}$ and $W\sim{Exp(10)}$, respectively. As can be seen the nonparametric estimator responded with a lower MSE.}
	
	\label{exp}
	
\end{table}

\begin{figure}[h!]
	\centering
	\includegraphics[width=0.8\textwidth]{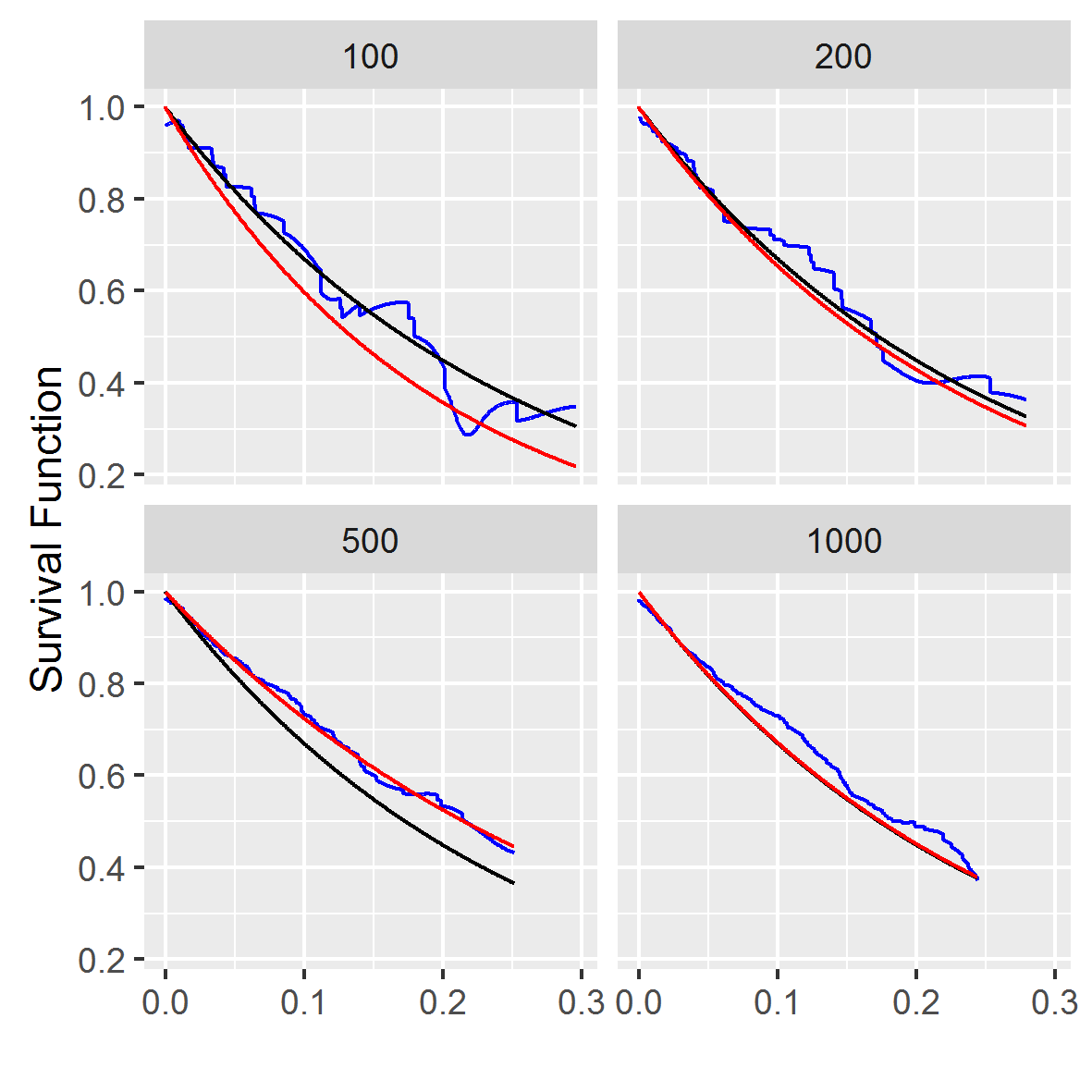}\protect\caption{Setting 1. \label{f_exp} The blue, red, and black curves represent the nonparametric, parametric, and true survival functions, respectively, for $N=100,200,500$ and $1000$.}
\end{figure}


We study the performance of both the parametric and nonparametric estimators that were proposed in Sections~\ref{sec:par} and~\ref{sec:nonpar}, respectively. Based on the setting of the case study discussed in Section \ref{sec:case}, we consider two simulation settings. In the case study, both the exponential and Weibull distributions seem to fit well the waiting time and patience time distributions, respectively. The case study data also indicated that the mean of the waiting time $W$ is smaller then the patience time $T$.     
Thus, the two simulation settings consist of samples from exponential and Weibull distributions in which the waiting time has a smaller mean then the patience time mean. 
In the first setting, a sample was taken from the model in which the patience time $T$ follows an exponential distribution with rate $4$, and the waiting time $W$ follows an exponential distribution with rate $10$. 
In the second setting a sample was taken from a model in which the patience time $T$ follows a Weibull distribution with rate $4$ and shape $2$, and the waiting time $W$ follows an exponential distribution with rate $10$. In both settings, the unknown probability of announcement is $q(t)=\exp(-12t)$. Taking the probability of announcement to be the increasing function $q(t)=1-\exp(-12t)$ or the constant function $q(t)=0.5$ yields similar results which are omitted. Moreover, we experimented with additional numerical values. The behavior and conclusions, as reported here, remain consistent across these experiments.

In each setting, we calculated the parametric estimator  for the rate of $T$ for five different sample sizes ($N=100,200,500,1000,2000$). For each sample size, we repeated the simulation $100$ times. When using the parametric method, it was assumed that both $T$ and $W$ follow an exponential distribution with unknown parameters. Note that this assumption holds for the first setting but does not hold for the second one. In other words, the second setting is carried out under a misspecified model. The results are shown in Figure~\ref{figure:sim-par}.

\begin{figure}[t!]
	\centering
	\includegraphics[width=0.8\textwidth]{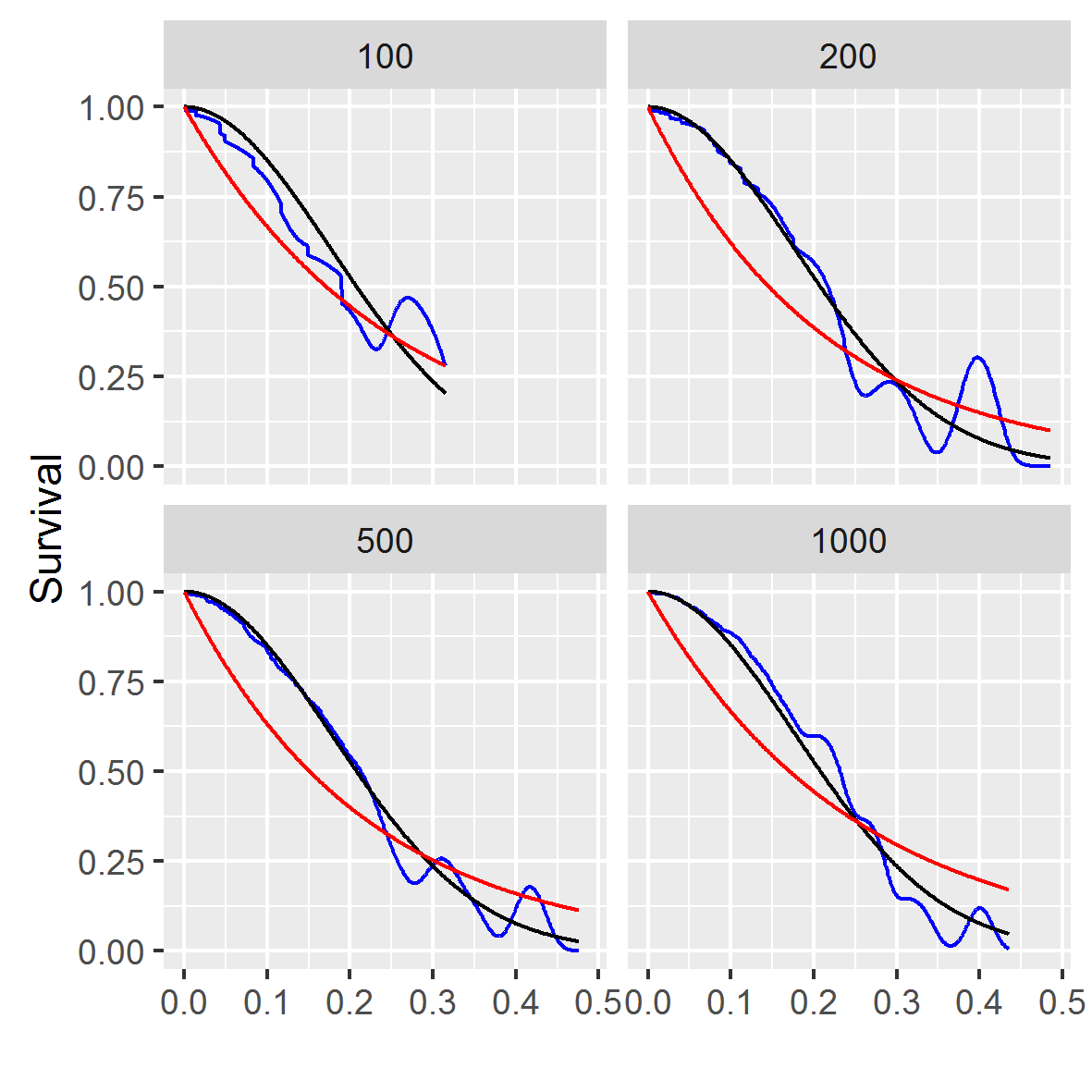}\protect\caption{Second setting. \label{f_wei} The blue, red, and black curves represent the nonparametric, parametric, and true survival functions, respectively.}
\end{figure}

We compare $\widehat{\overline{F}}_{n}$, the estimator of the survival function of $T$,  to the true survival function~$\overline{F}_{0}$. For the parametric estimation, $\widehat{\overline{F}}_{n}(t)=\exp(-\hat\theta{t})$, while for the nonparametric estimator $\widehat{\overline{F}}_{n}(t)$ is given by~\eqref{eq:nonparametric_estimator}. The comparison is done using mean square error (MSE), which is defined by \begin{align*}MSE(\widehat{\overline{F}}_{n},\overline{F}_{0})\equiv\int_{-\infty}^{\infty}{\left\lbrace \widehat{\overline{F}}_{n}(t)-\overline{F}_{0}(t)\right\rbrace ^{2}}f_{0}(t)dt\,,
\end{align*}
where $f_0$ is the density of $T$. The parametric and nonparametric survival function estimators are demonstrated in Figures \ref{f_exp} and \ref{f_wei}. Figure \ref{f_exp} represents the results of the first setting in which $T\sim{Exp(4)}$ and $W\sim{Exp(10)}$. Figure \ref{f_wei} represents the results of the second setting in which $T\sim{Weibull(4,2)}$ and $W\sim{Exp(10)}$.  Summaries of the MSE are given in Table~\ref{exp}. Not surprisingly, for Setting 1, since the parametric model is correct, the MSE is smaller for the parametric estimator. Similarly, since in Setting 2 the parametric model is incorrect, the MSE is smaller for the nonparametric estimator.

	\section{Case study}\label{sec:case}
\begin{figure}[t!]
	\centering
	\includegraphics[width=\textwidth]{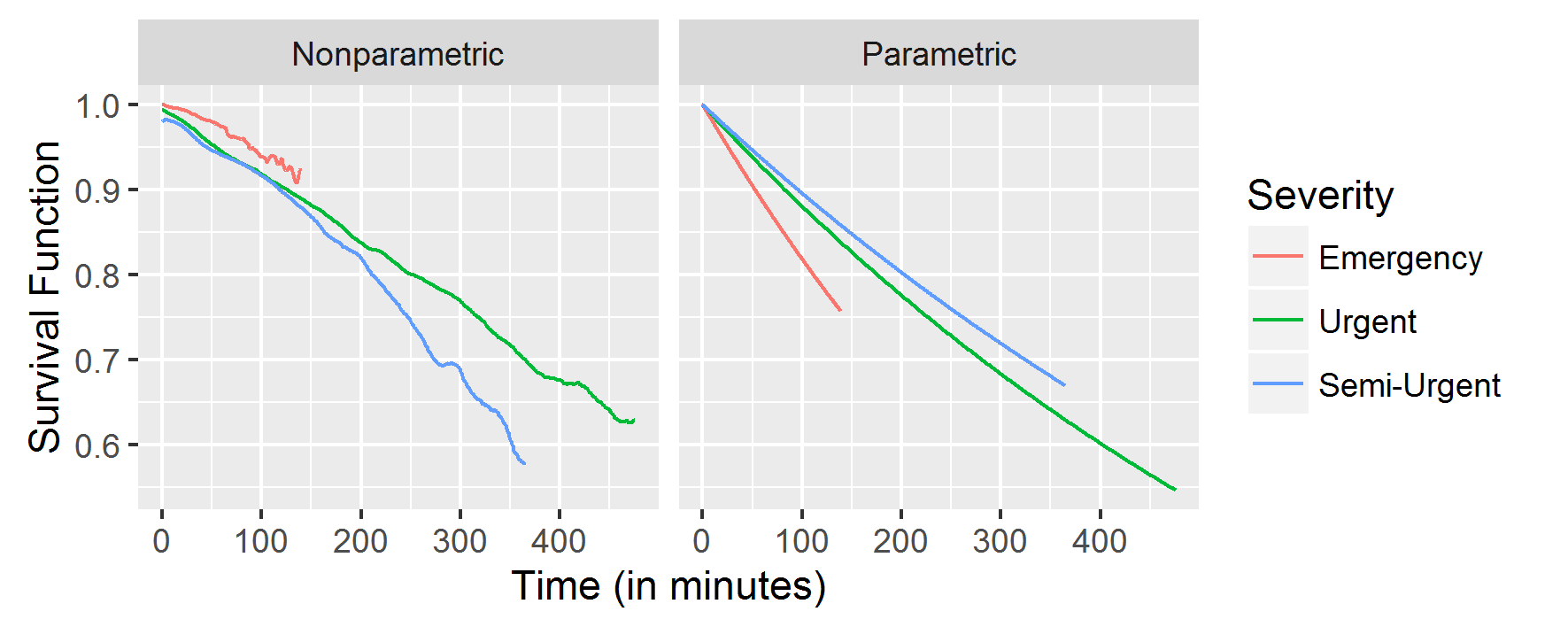}\protect\caption{Compression of the estimator for the survival function of the patience time at the three different levels of severity.\label{com_dat_1}}\end{figure}

\begin{figure}[h!]
	\centering
	\includegraphics[width=\textwidth]{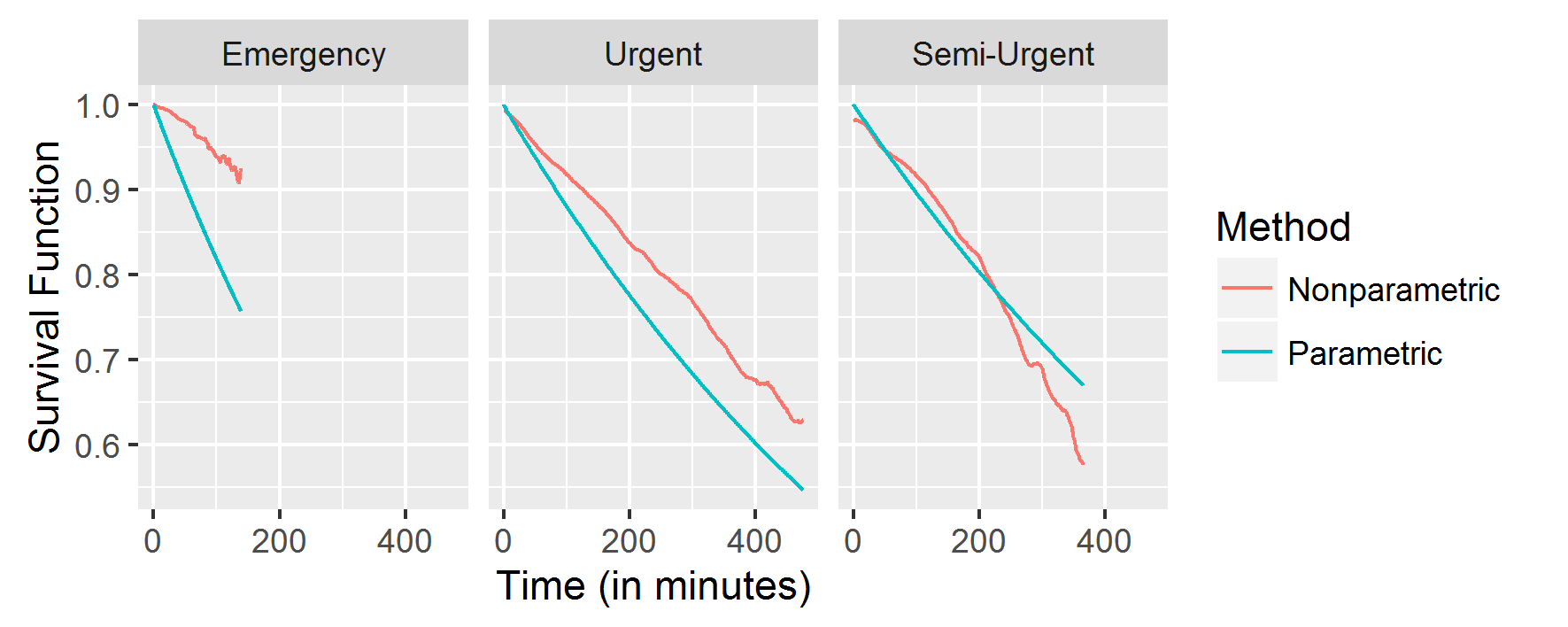}\protect\caption{Compression of the nonparametric and parametric estimators for the survival of the patience time by different levels of acuity.\label{com_dat_2}}
\end{figure}

Retrospective data were collected from all patient presentations to triage at an urban,
academic, adult-only emergency department (ED) with visits in calendar year 2008.
This data was used for the analysis in \citet{wiler2013emergency}.
The data consist of the waiting time of patients arriving at emergency rooms. One of the categories defined in this data is acuity. Since our model assumes that all patients follow the same distribution, we calculated the estimators for each level of acuity separately. We focused on the following three levels of acuity: emergency, urgent, and semi-urgent. The emergency level consist of $15106$ patients, the urgent level consist of $43600$ patients, and the semi-urgent level consist of $26541$ patients.

The data consists of the triple variables $(U_{i},\Delta_{i},Y_{i})$ described in Section \ref{sec:Model}. At each acuity level, an observation is categorized to one of the three possible categories.

Parametric and nonparametric estimators for the survival of each acuity level were calculated. The results of these estimators are given in
Figures \ref{com_dat_1} and \ref{com_dat_2}.
As can be seen in Figure 4, the nonparametric estimators of the patience time are stochastically ordered by levels of acuity. In other words, patients at the severe acuity level are less probable to loose patience than patients at the urgent level, who in turn are less prone to lose patience than patients at the semi-urgent level. The results for the parametric estimator seem unreasonable since one would expect that patients with more severe acuity level are more likely to loose patience, and loose it faster. 

\section{Discussion}\label{sec:dis}
In this paper, we consider survival data that combine observed, right-censored, and left-censored data. The setting we analyzed was that of patients who wait for treatment in an emergency department, where some patients may leave without being seen. We proposed both parametric and nonparametric estimators for the distribution of the patience time. Using simulation, we showed that when the parametric model holds, the parametric estimator estimates the patience time well. However, when the model is misspecified, the nonparametric estimator behaved better. In our case study, we also observed that the nonparametric estimator performed better.

So far, no baseline covariates were given. Novel parametric and nonparametric estimators are needed
for addressing settings that include baseline covariates.

\section{Acknowledgement}\label{sec:ack}
Y. Ritov was partially supported by the Israeli Science Foundation (grant No. 1770/15). 

Y. Goldberg was partially supported by the Israeli Science Foundation (grant No. 849/17). 
\appendix

\section{Proofs}
\subsection{\label{l:2}Proof of Lemma~\ref{lem:probablities_connections}}
\begin{align*}
pr(U\leq{t},\mathcal{C}=1)=&pr(U\leq{t},W\leq{T})
\\
=&pr(W\leq{t},W\leq{T})
\\
=&\int_{0}^{t}pr(W\leq{T}\mid{W=s})g(s)ds
\\=&\int_{0}^{t}pr(s\leq{T})g(s)ds
\\
=&\int_{0}^{t}g(s)\overline{F}(s)ds,\
\end{align*}
where in the fourth equality we use the independence between $W$ and $(Y,T)$. 

This establishes \ref{a}).	For \ref{b}), we have
\begin{align*}
pr(U\leq{t},\mathcal{C}=2)=&pr(U\leq{t},Y=1,T<W)
\\
=&pr(T\leq{t},Y=1,T<W)
\\
=&\int_{0}^{t}pr(s<W\mid{Y}=1,T=s)q(s)f(s)ds
\\
=&\int_{0}^{t}pr(s<W)q(s)f(s)ds		
\\
=&\int_{0}^{t}q(s)f(s)\overline{G}(s)ds,\end{align*}
\\
where in the fourth equality we use the independence between $W$ and $(Y,T)$.

Finally, for \ref{c}), \begin{align*}P(U\leq{t},\mathcal{C}=3)
=&pr(U\leq{t},Y=0,T<W)
\\=&pr(W\leq{t},Y=0,T<W)
\\=&\int_{0}^{t}pr(Y=0,T<s\mid{W=s})g(s)ds
\\=&\int_{0}^{t}g(s)pr(Y=0,T<s)ds
\\=&\int_{0}^{t}g(s)\int_{0}^{s}pr(Y=0\mid{T=x})f(x)dxds
\\=&\int_{0}^{t}g(s)\int_{0}^{s}\left\lbrace{}1-q(x)\right\rbrace{}f(x)dxds,\end{align*}	
where in the fourth equality we use the independence between $W$ and $(Y,T)$.	

\subsection{\label{p:4} Proof of  Theorem~\ref{Thm1}}

The log of the full likelihood is
\begin{align*}&
\l(D;\theta,\gamma)
=\sum_{i=1}^{n}\Bigg[{\left(1-\Delta_{i}\right)\left(\log{g(U_{i};\gamma)}+
	\log{\overline{F}(U_{i};\theta)}\right)}
\\
&\qquad\qquad\qquad+\Delta_{i}Y_{i}
\left(\log{q(U_{i})}+\log{f(U_{i};\theta)}+\log{\overline{G}(U_{i};\gamma)}\right)
\\
&
\qquad\qquad\qquad\qquad+\Delta_{i}(1-Y_{i}) \left\lbrace \log{g(U_{i};\gamma)}+\log{\int_{0}^{U_{i}}(1-{q(s))f(s;\theta)ds}}\right\rbrace \Bigg].
\end{align*}
Given the data $D$,
\begin{equation}\label{proof}\begin{split}
\frac{1}{n}l(D;\theta,\gamma)&=\mathbb{P}_{n}\left\lbrace{m}_{\gamma}(U,\Delta,Y)+c(U,\Delta,Y;\theta)\right\rbrace
\equiv\frac{1}{n}\sum_{i=1}^{n}\left\lbrace{}m_{\gamma}(U_{i},\Delta_{i},Y_{i})+c(U_{i},\Delta_{i},Y_{i};\theta)\right\rbrace
\end{split}\end{equation}
where $m_{\gamma}:\mathbb{R}^{+}\times{\{0,1\}^{2}}\rightarrow{\mathbb{R}}$ is defined by
\begin{align*}   m_{\gamma}(u,\delta,y)\equiv(1-\delta)\log{g(u;\gamma)}+\Delta{y}\log{\overline{G}(u;\gamma)}+\delta(1-y)\log{g(u;\gamma)}.\end{align*}
and \begin{align*}c(u,\delta,y;\theta)\equiv&(1-\delta)\log{\overline{F}(u;\theta)}\\&+\delta{y}\left(\log{q(u)}+\log{f(u;\theta)}\right)+\delta(1-y)\log{\int_{0}^{u}\left\lbrace{}1-q(s)\right\rbrace{}f(s;\theta)ds.}\end{align*}
From assumption A\ref{A1} we obtain that, for each $\theta\in\Theta$, 	$\arg\max_{\gamma\in\Gamma}l(D;\gamma,\theta)=\arg\max_{\gamma\in\Gamma}\mathbb{P}_{n}(m_{\gamma})$. The $\gamma$ that maximizes $L(D;\theta,\gamma)$ does not depend on the value of $\theta$ or the function $p$.
Define $M_{n}(\gamma)\equiv\mathbb{P}_{n}m_{\gamma}$ and  $M(\gamma)\equiv{P}m_{\gamma}$.If, for a general function $h$, $Ph\equiv\int{h(x)dP(x)}$ and $\mathbb{P}_{n}h\equiv{n^{-1}}\sum_{i=1}^{n}{h(X_{i}}$) then by Assumptions~A\ref{A1}--A\ref{A3}, Theorem 5.7 in~\citet{VDV98} can be applied. Therefore
$\widehat{\gamma}_{n}\rightarrow{\gamma_{0}}$, in probability, which concludes the proof of  \ref{enum:consistent_gamma}).

Given the data $D$, the term $\dldp$ is a function of $\gamma$ and does not depend on the unknown function $p$. We also have
\begin{align*}
\frac{1}{n}\dldp=\mathbb{P}_{n}\psi_{\gamma}(U,\Delta,Y)\equiv\frac{1}{n}\sum_{i=1}^{n}{\psi_{\gamma}(U_{i},\Delta_{i},Y_{i}),}
\end{align*}
where $\psi_{\gamma}:\mathbb{R}^{+}\times{\{0,1\}^{2}}\rightarrow{\mathbb{R}}$ is defined as  \begin{align*}\psi_{\gamma}(u,\delta,y)\equiv(1-\delta)\frac{\frac{\partial}{\partial\gamma}g(u;\gamma)}{g(u;\gamma)}-\delta{y}\frac{\frac{\partial}{\partial\gamma}\overline{G}(u;\gamma)}{\overline{G}(u;\gamma)}+\delta(1-y)\frac{\frac{\partial}{\partial\gamma}g(u;\gamma)}{g(u;\gamma)}. \end{align*}

By Assumptions~(A\ref{A1})--(A\ref{A3}),  $\psi_{\gamma}$ satisfies the conditions of Theorem 5.41 in \citet{VDV98}  and, therefore,
\begin{align*}\sqrt{n}(\hat{\gamma}_{n}-\gamma_{0})=-(P\dot{\psi}_{\gamma_{0}})^{-1}\frac{1}{\sqrt{n}}\sum_{i=1}^{n}\psi_{\gamma_{0}}(U_{i},\Delta_{i},Y_{i})+o_{p}(1),\end{align*}
where $\dot{\psi}_{\gamma}(x)=\frac{\partial}{\partial{\gamma}}\psi_{\gamma}(x)$. Hence $\hat{\gamma}_{n}$ is a linear asymptotically normal (LAN) estimator with influence function $\varphi\equiv{-}(P\dot{\psi}_{\gamma_{0}})^{-1}\psi_{\gamma_{0}}$ . From all of the above we get that \ref{enum:normal_gamma}) is proved with $V_{\gamma_{0}}=(P\dot{\psi}_{\gamma_{0}})^{-1}{P\psi_{\gamma_{0}}\psi_{\gamma_{0}}^{t}}(P\dot{\psi}_{\gamma_{0}})^{-1}$.
\\To prove \ref{enum:consistent_theta}), note that due to the term $\log{\int_{0}^{U_{i}}(1-q(s))f(s;\theta)ds}$ that appears in  $l(D;\theta,\gamma)$, the term $\frac{\partial{l(D;\theta,\gamma)}}{\partial{\theta}}$  depends on the unknown function $p$. We therefore consider a partial likelihood function such that its derivate with respect to $\theta$ does not depend on $p$. The partial likelihood that satisfies this request is the partial likelihood of $\mathcal{C}=1$:
\begin{align*}\prod_{i=1}^{n}\left\lbrace\frac{g(U_{i};\gamma)\overline{F}(U_{i};\theta)}{\int_{0}^{\infty}g(s;\gamma)\overline{F}(s;\theta)ds}\right\rbrace{}^{1-\Delta_{i}}
\end{align*}
The log of the partial likelihood is
\begin{align*}
\l_{partial}(D;\theta,\gamma)=\sum_{i=1}^{n}\left(1-\Delta_{i}\right)\left\lbrace\log{g(U_{i};\gamma)}+\log{\overline{F}(U_{i};\theta)}-\log{\int_{0}^{\infty}g(s;\gamma)\overline{F}(s;\theta)ds}\right\rbrace.
\end{align*}
Given the data $D$, the term $l_{partial}(D;\theta,\gamma)$ is a function only of the parameters $\theta$ and $\gamma$. We also have
\begin{align*}\frac{1}{n}\l_{parital}(D;\theta,\gamma)=\mathbb{P}_{n}r_{\theta,\gamma}(U,\Delta,Y)\equiv\frac{1}{n}\sum_{i=1}^{n}{r_{\theta,\gamma}(U_{i},\Delta_{i},Y_{i})},
\end{align*}
where $r_{\theta,\gamma}:\mathbb{R}^{+}\times{\{0,1\}^{2}}\rightarrow{\mathbb{R}}$ is given by
\begin{align*}    r_{\theta,\gamma}(U,\Delta,Y)\equiv\left(1-\Delta\right)\left\lbrace\log{g(U;\gamma)}+\log{\overline{F}(U;\theta)}-\log{\int_{0}^{\infty}g(s;\gamma)\overline{F}(s;\theta)ds}\right\rbrace.
\end{align*}
Define $M_{n}(\theta,\gamma)\equiv\mathbb{P}_{n}r_{\theta,\gamma}$, and $M(\theta,\gamma)\equiv{P}r_{\theta,\gamma}$. Then,   
Theorem 5.7 in \cite{VDV98} can be applied. Therefore $\left(\widehat{\theta}_{n},\widehat{\gamma}_{n}\right)\rightarrow{\left(\theta_{0},\gamma_{0}\right)}$ in probability, and
in particular $\widehat{\theta}_{n}\rightarrow{\theta_{0}}$ in probability, and \ref{enum:consistent_theta}) is proven.
\\In order to prove \ref{enum:normal_theta}), note that
\begin{align*}\frac{1}{n}\dldt=\mathbb{P}_{n}\phi_{\theta,\gamma}(U,\Delta,Y)\equiv\frac{1}{n}\sum_{i=1}^{n}{\phi_{\theta,\gamma}(U_{i},\Delta_{i},Y_{i})},\end{align*}
where $\phi_{\theta,\gamma}:\mathbb{R}^{+}\times{\{0,1\}^{2}}\rightarrow{\mathbb{R}}$ is defined as
\begin{align*}\phi_{\theta,\gamma}(u,\delta,y)\equiv(1-\delta)\left\lbrace\frac{\frac{\partial}{\partial{\theta}}\overline{F}(u;\theta)}{\overline{F}(u;\theta)}-\frac{\frac{\partial}{\partial\theta}\int_{0}^{\infty}g(s;\gamma)\overline{F}(s;\theta)ds}{\int_{0}^{\infty}g(s;\gamma)\overline{F}(s;\theta)ds}\right\rbrace.\end{align*}

Using Assumptions~A\ref{A1}--A\ref{A3}, together from Theorem 5.41 in \citet{VDV98}, we obtain that
\begin{align*}\sqrt{n}(\hat{\gamma}_{n}-\gamma_{0})=\frac{1}{\sqrt{n}}\sum_{i=1}^{n}\left\lbrace-(P\dot{\psi}_{\gamma_{0}})^{-1}\psi_{\gamma_{0}}(U_{i},\Delta_{i},Y_{i})\right\rbrace+o_{p}(1).\end{align*}
Define $\Phi_{n}(\theta,\gamma)\equiv\mathbb{P}_{n}\phi_{\theta,\gamma}$ and note that $\Phi(\theta_{0},\gamma_{0})\equiv{P}\phi_{\theta_{0},\gamma_{0}}=0$ (since under the true parameters $P\phi_{\theta_{0},\gamma_{0}}={\frac{\partial}{\partial{\theta}}}\int{dh_{0}}={\frac{\partial}{\partial{\theta}}}1=0$, where $dh_{0}$ is the true distribution of category 1).

By Talyor's theorem,
\begin{align*} &0=\Phi_{n}(\widehat{\theta}_{n},\widehat{\gamma}_{n})
=\Phi_{n}(\theta_{0},\gamma_{0})+\left\lbrace\frac{\partial}{\partial\theta}\Phi_{n}(\theta_{0},\gamma_{0})\right\rbrace^{T}(\widehat{\theta}_{n}-\theta_{0})
+\left\lbrace\frac{\partial}{\partial\phi}\Phi_{n}(\theta_{0},\gamma_{0})\right\rbrace^{T}(\widehat{\gamma}_{n}-\gamma_{0})+o_{p}\left(n^{-{1}/{2}}\right)
\\
\Rightarrow &{0=\sqrt{n}\Phi_{n}(\theta_{0},\gamma_{0})+\left\lbrace\frac{\partial}{\partial\theta}\Phi_{n}(\theta_{0},\gamma_{0})\right\rbrace^{T}\sqrt{n}\left(\widehat{\theta}_{n}-\theta_{0}\right)+\left\lbrace\frac{\partial}{\partial\gamma}\Phi_{n}(\theta_{0},\gamma_{0})\right\rbrace^{T}\sqrt{n}\left\lbrace\widehat{\gamma}_{n}-\gamma_{0}\right\rbrace+o_{p}(1)}.
\\
&=\sqrt{n}\Phi_{n}(\theta_{0},\gamma_{0})+\left\lbrace\frac{\partial}{\partial\theta}\Phi_{n}(\theta_{0},\gamma_{0})\right\rbrace^{T}\sqrt{n}\left(\widehat{\theta}_{n}-\theta_{0}\right)
\\
&\quad -\left\lbrace\frac{\partial}{\partial\gamma}\Phi_{n}(\theta_{0},\gamma_{0})\right\rbrace^{T}(P\dot{\psi}_{\gamma_{0}})^{-1}\frac{1}{\sqrt{n}}\sum_{i=1}^{n}\psi_{\gamma_{0}}(U_{i},\Delta_{i},Y_{i})+o_{p}(1)
\\
&=\frac{1}{\sqrt{n}}\sum_{i=1}^{n}\left[\phi_{\theta_{0},\gamma_{0}}(U_{i},\Delta_{i},Y_{i})-\left\lbrace\frac{\partial}{\partial\gamma}\Phi_{n}(\theta_{0},\gamma_{0})\right\rbrace^{T}(P\dot{\psi}_{\gamma_{0}})^{-1}\psi_{\gamma_{0}}(U_{i},\Delta_{i},Y_{i})\right]
\\&+\left\lbrace\frac{\partial}{\partial\theta}\Phi_{n}(\theta_{0},\gamma_{0})\right\rbrace^{T}\sqrt{n}(\widehat{\theta}_{n}-\theta_{0})+o_{p}(1).
\end{align*}
Elementary arithmetic leads to
\begin{align*}
& \sqrt{n}\left(\widehat{\theta}_{n}-\theta_{0}\right)=-\left(EE^{T}\right)^{-1}\frac{1}{\sqrt{n}}\sum_{i=1}^{n}\left\lbrace\phi_{\theta_{0},\gamma_{0}}(U_{i},\Delta_{i},Y_{i})-B^{T}(P\dot{\psi}_{\gamma_{0}})^{-1}\psi_{\gamma_{0}}(U_{i},\Delta_{i},Y_{i})\right\rbrace+o_{p}(1),
\end{align*}
where $B\equiv\frac{\partial}{\partial\gamma}\Phi(\theta_{0},\gamma_{0})$, and $E\equiv\frac{\partial}{\partial\theta}\Phi(\theta_{0},\gamma_{0})$.
Hence, $\widehat{\theta}_{n}$ is a LAN estimator with the influence function
\begin{align*}\varphi=-\left(EE^{T}\right)^{-1}\frac{1}{\sqrt{n}}\left\lbrace\phi_{\theta_{0},\gamma_{0}}-B^{T}(P\dot{\psi}_{\gamma_{0}})^{-1}\psi_{\gamma_{0}}\right\rbrace.\end{align*} Summarizing,

$\sqrt{n}(\widehat{\theta}_{n}-\theta_{0})\rightarrow{N(0,S_{\theta_{0},\gamma_{0}})}$ in distribution, where
$S_{\theta_{0},\gamma_{0}}=P\varphi\varphi^{T}$, hence, \ref{enum:normal_theta}) is proven.

\subsection{\label{p:3}Proof of Lemma~\ref{lem:estimating}}

\begin{proof}
	We use similar arguments to those in the proof of the convergence of the Nelson--Aalen estimator to a cumulative hazard function \citep[see][page 240]{Kosorok08}. Hence we have that
	
	\begin{align*}\sqrt{n}\begin{Bmatrix}\widehat{N}_{n}(t)-N(t)\\\widehat{Y}_{n}(t)-Y(t)\end{Bmatrix}=O_{p}(1),\end{align*}
	
	where $N(t)=pr(Y=1, T\leq{W},T\leq{t})\quad \text{and}\quad Y(t)=pr(U\geq{t})$.

	Since, by Section \ref{sec:Model},  
	\begin{align*}\{U>t\}=\{Y=0,W>t\}\cup{\{Y=1,W>t, T>t\}}.\end{align*}
	Hence,
	\begin{align*}pr(U>t)&=pr(Y=0,W>t)+pr(Y=1,W>t,T>t)\\&=pr(W>t)\left\lbrace{}pr(Y=0)+pr(Y=1,T>t)\right\rbrace\\&=pr(W>t)\big(pr(Y=0)+pr(Y=1)-pr(Y=1,T\leq{t})\big)\\&=pr(W>t)\left\lbrace{}1-pr(Y=1,T\leq{t})\right\rbrace\\&=\overline{G}(t)\left\lbrace{}1-\int_{0}^{t}{q(s)f(s)ds}\right\rbrace,\end{align*}
	where in the second equality we use the independence between $W$ and $(Y,T)$. 
	
	By Lemma~\ref{lem:probablities_connections} \begin{align*}pr(Y=1,T\leq{W},T\leq{t})=\int_{0}^{t}{q(s)f(s)\overline{G}(s)}ds.\end{align*}
	Using the continuity of the derivative operator and of the integral operator, we get that  
	\begin{align*}\widehat{D}_{n}(t)\rightarrow \int_{0}^{t}\frac{q(s)f(s)\overline{G}(s)}{\overline{G}(s)\left\lbrace{}1-\int_{0}^{s}{q(x)f(x)ds}\right\rbrace}ds\end{align*}
	in probability.

	Note that
	\begin{align*}&\int_{0}^{t}{\frac{\overline{G}(s)q(s)f(s)}{\overline{G}(s)\left\lbrace{}1-\int_{0}^{s}{q(x)f(x)dx}\right\rbrace}ds}=\int_{0}^{t}{\frac{q(s)f(s)}{\left\lbrace{}1-\int_{0}^{s}{q(x)f(x)dx}\right\rbrace}ds}
	\\
	&\quad =-\int_{0}^{t}{\frac{\partial}{\partial{s}}\log{\left\lbrace{}1-\int_{0}^{s}{q(x)f(x)dx}\right\rbrace{}}ds}=-\log{\left\lbrace{}1-\int_{0}^{t}{q(s)f(s)ds}\right\rbrace{}}.\end{align*}
	Hence, by the delta method \citep[see][Chapter 12.2.2.2]{Kosorok08}, we get that\begin{align*}\widehat{D}_{n}(t)\rightarrow -\log{\left\lbrace{}1-\int_{0}^{t}{q(s)f(s)ds}\right\rbrace{}}\end{align*} in probability, with convergence at rate $n^{1/2}$.
	
	Since $y=-\log{(1-x)}\Leftrightarrow{x=1-\exp{(-y)}}$ and
	by the continuous mapping theorem \citep[see][Theorem 7.7]{Kosorok08}, we get  that $\widehat{A}(t)=1-\exp(-\widehat{D}_{n}(t))$ is an estimator of $\int_{0}^{t}{q(s)f(s)ds}$, at the rate of $n^{1/2}$ as desired.
\end{proof}

\subsection{\label{p:2}Proof of Theorem~\ref{thm2}}
For the proof of Theorem~\ref{thm2}, we need the following lemma, which is elementary hence stated without proof.
\begin{lem}\label{lem:aux}
	Let $(a_{n})_{n=1}^{\infty}, (b_{n]})_{n=1}^{\infty}$ be positive sequences. If $X_{n}-X=O_{p}(a_{n})$ and        \\$Y_{n}-Y=O_{p}(b_{n})$, as well as $P(\lvert{X}\rvert>l)=1$ for some $l>0$. Then we have:
	\begin{enumerate}[i)]
		\item\label{d} $X_{n}+Y_{n}-(X+Y)=O_{p}(a_{n}\vee{b_{n}}),$
		\item\label{e} $X_{n}Y_{n}-XY=O_{p}(a_{n}\vee{b_{n}}),$
		\item\label{g} $\frac{1}{X_{n}}-\frac{1}{X}=O_{p}(a_{n}).$
	\end{enumerate}
\end{lem}
\begin{proof}[Proof of Theorem~\ref{thm2}]
	Recall that	 \begin{align*}\label{eq:nonparametric_estimator}
	\Fn_{n}(t)=\frac{\widehat{pr}(Y=0,W<T){\rn}_{3}(t)+\widehat{pr}(T\leq{W}){\rn}_{1}(t)\widehat{A}(t)}{\widehat{pr}(Y=0,W<T){\rn_{3}}(t)+\widehat{pr}(T\leq{W})\rn_{1}(t)}\end{align*} is an estimator of $F(t)$.
	\\	For all $t>0$,
	
	\begin{align*}
	\widehat{pr}(Y=0,W<T)-{pr}(Y=0,W<T)=O_{p}(n^{-{1}/{2}})\,,
	\end{align*}
	and
	\begin{align*}
	\widehat{pr}(T\leq{W})-{pr}(T\leq{W})=O_{p}(n^{-{1}/{2}})\,,
	\end{align*}
	as both are empirical distribution estimators. By Chapter~1.7 of \citet{tsybakov_introduction_2008}, \\for all $t>0$, ${\rn}_{j}(t)-r_{j}(t)=O_{p}(n^{-{\beta}/{(2\beta+1)}})$, for $j=1,3$.  By Lemma~\ref{lem:estimating},
	$\widehat{A}(t)-A(t)=O_{p}(n^{-{1}/{2}})$.
	
	By Lemma~\ref{lem:aux}.\ref{e}) 
	\begin{align*}
	\widehat{pr}(Y=0,W<T){\rn}_{3}(t)-{pr}(Y=0,W<T){\rn}_{3}(t)=O_{p}(n^{-{\beta}/{(2\beta+1})})\,,
	\end{align*}
	and
	\begin{align*}
	&\widehat{pr}(T\leq{W}){\rn}_{1}(t)\widehat{A}(t)-{pr}(T\leq{W}){\rn}_{1}(t){A}(t)=O_{p}(n^{-{\beta}/{(2\beta+1)}})\,.
	\end{align*}
	
	Therefore, by Lemma~\ref{lem:aux}.\ref{d}),
	\begin{align*}
	&\widehat{pr}(Y=0,W<T){\rn}_{3}(t)+\widehat{pr}(T\leq{W}){\rn}_{1}(t)\widehat{A}(t)-{pr}(Y=0,W<T){r}_{3}(t)-{pr}(T\leq{W}){r}_{1}(t){A}(t)\\&=O_{p}(n^{-{\beta}/{(2\beta+1)}})\,.
	\end{align*}
	Similarly, 
	\begin{align*}&\widehat{pr}(Y=0,W<T){\rn}_{3}(t)+\widehat{pr}(T\leq{W}){\rn}_{1}(t)-{pr}(Y=0,W<T){r}_{3}(t)-{pr}(T\leq{W}){r}_{1}(t)
	\\
	&=O_{p}(n^{-{\beta}/{(2\beta+1)}}).\end{align*}
	By Lemma~\ref{lem:aux}.\ref{g}),
	\begin{align*}
	&\frac{1}{\widehat{pr}(Y=0,W<T){\rn}_{3}(t)+\widehat{pr}(T\leq{W}){\rn}_{1}(t)}-\frac{1}{{pr}(Y=0,W<T){r}_{3}(t)+{pr}(T\leq{W}){r}_{1}(t)}
	\\
	&=O_{p}(n^{-{\beta}/{(2\beta+1)}}).\end{align*}
	By Lemma~\ref{lem:aux}.\ref{d}) again,
	\begin{align*}
	&\frac{\widehat{pr}(Y=0,W<T){\rn}_{3}(t)+\widehat{pr}(T\leq{W}){\rn}_{1}(t)\widehat{A}(t)}{\widehat{pr}(Y=0,W<T){\rn_{3}}(t)+\widehat{pr}(T\leq{W})\rn_{1}(t)}-\frac{{pr}(Y=0,W<T){r}_{3}(t)+{pr}(T\leq{W}){r}_{1}(t){A}(t)}{{pr}(Y=0,W<T){r_{3}}(t)+{pr}(T\leq{W})r_{1}(t)}
	\\
	&=O_{p}(n^{-{\beta}/{(2\beta+1)}}).
	\end{align*}
	In other words, $\Fn_{n}(t)-F(t)=O_{p}(n^{-{\beta}/{(2\beta+1)}})$, which complete the proof of Theorem~\ref{thm2}.
\end{proof}

\bibliographystyle{plainnat}

\end{document}